\documentclass[11pt]{amsart}
\usepackage{latexsym,amsfonts,amsthm,amssymb}

\newtheorem{thm}{Theorem}
\newtheorem{lem}[thm]{Lemma}

\newtheorem{cor}[thm]{Corollary}

\theoremstyle{remark}
\newtheorem{rem}{Remark}

\theoremstyle{definition}

\newcommand{\N}{\mathbb{N}}

\newcommand{\calP}{\mathcal{P}}

\newcommand{\eps}{\varepsilon}
\newcommand{\PP}{\mathbb{P}}

\DeclareMathOperator{\disp}{disp}

\title{An upper bound on the minimal dispersion}
\author{Mario Ullrich}
\address{Institut f\"ur Analysis, Johannes Kepler Universit\"at, Altenberger Str. 69, 4040 Linz, Austria}
\email{mario.ullrich@jku.at}
\author{Jan Vyb\'iral}
\address{Dept. of Mathematics FNSPE, Czech Technical University in Prague, Trojanova 13, 12000 Prague, Czech Republic}
\thanks{Both authors acknowledge the kind support of the Oberwolfach Research Institute for Mathematics during the Oberwolfach Workshop ``Perspectives in High-Dimensional Probability and Convexity''.
The second author was supported by the ERC CZ grant LL1203 of the Czech Ministry of Education and by the Neuron Fund for Support of Science.}
\email{jan.vybiral@fjfi.cvut.cz}

\begin{document}

\begin{abstract}
For $\eps\in(0,1/2)$ and a natural number $d\ge 2$, let $N$ be a natural number with
\[
N \,\ge\, 2^9\,\log_2(d)\, \left(\frac{\log_2(1/\eps)}{\eps}\right)^2.
\] 
We prove that there is a set of $N$ points in the unit cube $[0,1]^d$,
which intersects all axis-parallel boxes with volume $\eps.$
That is, the dispersion of this point set is bounded from above by $\eps$.
\end{abstract}

\maketitle

\section{Introduction}

We are interested in bounds on the volume of the largest axis-parallel box that 
does not contain any point from a given finite point set $\calP\subset[0,1]^d$. 
Moreover, we would like to find a point set such that this volume is as small 
as possible. 
To be precise, we define, for $d\in\N$ and a point set $\calP\subset[0,1]^d$, 
the \emph{dispersion of $\calP$} by
\[
\disp(\calP) \;:=\; \sup_{B\colon B\cap\calP=\varnothing}\, |B|, 
\]
where the supremum is over all axis-parallel boxes $B=I_1\times\dots\times I_d$ 
with intervals $I_\ell\subset[0,1]$, 
and $|B|$ denotes the (Lebesgue) volume of $B$.
Moreover, for $n,d\in\N$, let the \emph{$n$th-minimal dispersion} 
be defined by 
\[
\disp(n,d) \;:=\; \inf_{\substack{\calP\subset[0,1]^d\colon\\ \#\calP=n}}\, \disp(\calP)
\]
and define its inverse function
\[
N(\eps,d) \;:=\; \min\Bigl\{n\colon \disp(n,d)\le\eps\Bigr\}.
\]

These quantities were introduced by Rote and Tichy~\cite{RT96} 
(as a modification of a quantity considered by Hlawka~\cite{Hl76}) 
and attracted quite a lot of attention in the past years 
in the context of information-based complexity theory, where the explicit 
dependence of certain geometric quantities on the dimension $d$ plays 
a crucial role.
Bounds on the dispersion (or any of its variants) translate into bounds on 
worst-case errors (and hence complexity bounds) for several numerical problems. 
These include optimization in different settings~\cite{Niederreiter83,YLV00}, 
approximation of high-dimensional rank-1 tensors~\cite{BDDG14,NR15}
and, very recently, approximation of $L_p$-norms and 
Marcinkiewicz-type discretization~\cite{Te17a,Te17b,Te17c}.
However, it is still not clear so far, if there exists a numerical problem that corresponds to the dispersion
in the same way as the \emph{discrepancy} corresponds to numerical integration, 
see e.g.~\cite{DP10,DP14,DT97,Niederreiter92,No15,NW10}.

Besides this, the dispersion is clearly an interesting geometric quantity 
on its own.
It is easy to define and one might think it is also simple to tackle.   
But, as the dispersion still resists a precise analysis,
this does not seem to be the case. 
However, there are several upper and lower bounds on the minimal dispersion, 
most of which were established in the past three years. 
Here we comment briefly on the state of the art.

First of all, it is quite easy to see that the minimal dispersion is of order 
$n^{-1}$ for all $d$. The best bounds of this order so far, 
which show also an explicit dependence on $d$,  
are
\[
\frac{\log_2(d)}{4(n+\log_2(d))} \;\le\; \disp(n,d)
\;\le\; \frac{C^d}{n}
\]
for some constant $C<\infty$. The lower bound is due to 
Aistleitner et al.~\cite{AHR15} 
and the upper bound was obtained by Larcher~\cite{LaPC} 
(see \cite[Section~4]{AHR15} for the proof).
Concerning the dependence on the dimension $d$, we see that the above 
bounds are far from being tight. However, it was recently proved by 
Sosnovec~\cite{Sos17}, that (surprisingly) the logarithmic dependence 
in the lower bound is sharp. He proved that, for every fixed $\eps>0$,  
\[
N(\eps,d) \,\le\, c_\eps\,\log_2(d).
\]
However, in this bound the $\eps$-dependence is far off. 
Further results on the dispersion are polynomial (in $d$ and $1/\eps$) bounds 
by Rudolf~\cite{Ru17} (see Remark~\ref{rem:daniel}) and
an explicit construction based on sparse grids by Krieg~\cite{Kr17}.
Interestingly, a lower bound linear
in $d$ was recently obtained by one of the authors~\cite{MU18} in the periodic 
setting.

It seems reasonable to conjecture that $\disp(n,d)\asymp\log(d)/n$. 
However, it is not yet clear if this bound can hold for all $n$ and $d$.

In this article we refine the analysis of \cite{Sos17} paying attention to the $\eps$-dependence and narrow the existing gap. We prove an
upper bound on the inverse of the minimal dispersion that is
logarithmic in $d$ and almost quadratic in $1/\eps$.

\medskip
\begin{thm}\label{thm:main}
Let $d\ge 2$ be a natural number and let $\eps\in(0,1/2)$. Then there exists a point set 
$\calP\subset[0,1]^d$ with $\disp(\calP)\le\eps$ and 
\[
\#\calP \;\le\; 2^7\,\log_2(d)\,\frac{\bigl(1+\log_2(\varepsilon^{-1})\bigr)^2}{\varepsilon^{2}}.
\]
\end{thm}
\medskip

Clearly, the right hand side is bounded above by the $N$
given in the abstract.
Moreover, Theorem~\ref{thm:main} directly implies the following.


\begin{cor}\label{cor:main}
For $n,d\in\N$ with $n\ge2$ and $d\ge 2$ we have
\[
\disp(n,d) \;\le\; c\,\log_2(n)\,\sqrt{\frac{\log_2(d)}{n}}
\]
for some absolute constant $c>0$.
\end{cor}

\goodbreak


\section{Proof}

We will now prove Theorem~\ref{thm:main}. 
For this, we have to prove that there exists a point set $\calP$ with the 
desired cardinality that has dispersion bounded by $\eps$, i.e., 
every box of volume $\eps$ contains at least one point from $\calP$. 

For $0<\varepsilon<1/2$, let $k\in\N$ with $2^{-k}\le \varepsilon<2^{-k+1}$, 
i.e., $k=\lceil\log_2(1/\varepsilon)\rceil\ge2$, 
and define 
$$
M_k=\left\{\frac{1}{2^k},\frac{2}{2^k},\dots,\frac{2^k-1}{2^k}\right\}\subset [0,1].
$$
We consider the random point set $X=\{x^1,x^2,\dots,x^n\}\subset M_k^d$ with 
the coordinates $x^j_{\ell}$, $j=1,\dots,n$, $\ell=1,\dots,d$, 
being chosen independently and uniformly from $M_k$.
We show that for $n$ growing polynomially in $1/\varepsilon$ and 
logarithmically in $d$, 
$X$ intersects every cube with sides parallel to
the coordinate axis and volume at least $\varepsilon$ with positive probability. 
This proves our existence result.

We begin with splitting the set of all boxes of volume at least $2^{-k}$ into 
several groups. This is necessary for the following union bound.
Define 
\[
\Omega_k \;:=\; \Bigl\{B\subset[0,1]^d\colon B \ \text{is an axis-parallel box with}\ |B|>2^{-k}\Bigr\}
\]
and, for $p=(p_{\ell})_{l=1}^d\in M_k^d$ and $s=(s_{\ell})_{\ell=1}^d\in\{0,\dots,2^k-1\}^{d}$, 
let 
\begin{equation}\label{eq:defOmegak}
\begin{split}
\Omega_k(p,s) \,:=\, &\Bigl\{I_1\times\dots\times I_d\in\Omega_k\colon \quad \frac{s_\ell}{2^k}<|I_\ell|\le\frac{s_\ell+1}{2^k}\\
&\quad \text{and}\quad 
\inf I_\ell\in\Bigl[p_\ell-\frac1{2^k},p_\ell\Bigr)\quad\text{for all}\quad \ell=1\dots,d
\Bigr\}.
\end{split}\end{equation}
Clearly, the sets $\Omega_k(p,s)$ form a partition of $\Omega_k$. 
It will be important in the following that all sets from 
$\Omega_k(p,s)$ contain almost the same elements from $M_k^d$, and that 
$\Omega_k(p,s)=\varnothing$ for several choices of $s$ and $p$. 
E.g., this is the case if $s_\ell=0$ for some $\ell=1,\dots,d$.

The proof of the main result is based, inter alia, on the following lemma, 
which will be proved at the end of this section.

\medskip

\begin{lem}\label{lem:prob}
Let $x$ be uniformly distributed in $M_k^d$. 
Then, for each $B\in\Omega_k$, 
\[
\PP(x\in B) \;>\; 2^{-k-4}.
\]
Moreover, for each $p\in M_k^d$ and $s\in\{1,\dots,2^k-1\}^d$, we have 
\[
\PP\bigl(\exists B\in\Omega_k(p,s)\colon x\notin B \bigr) 
\;<\; \exp\bigl(-2^{-k-4} \bigr).
\]
\end{lem}
\medskip
\begin{proof}[Proof of Theorem~\ref{thm:main}]
Recall that $X=\{x^1,x^2,\dots,x^n\}\subset M_k^d$ is our random point set.
By a simple union bound and Lemma~\ref{lem:prob}, we obtain
\[\begin{split}
\PP\bigl(\exists B\in\Omega_k\colon X\cap B=\varnothing\bigr)
\;&\le\; \sum_{p,s\colon \Omega_k(p,s)\neq\varnothing} 
	\PP\bigl(\exists B\in\Omega_k(p,s)\colon X\cap B=\varnothing\bigr) \\
\;&=\; \sum_{p,s\colon \Omega_k(p,s)\neq\varnothing} 
	\PP\bigl(\exists B\in\Omega_k(p,s)\colon x^1\notin B \bigr)^n \\
\;&<\; \#\bigl\{(p,s)\colon \Omega_k(p,s)\neq\varnothing\bigr\} \,
	\exp\bigl(-n\,2^{-k-4} \bigr).
\end{split}\]

To estimate further, we need to bound from above the number of pairs $(p,s)\in M_k^d\times \{0,\dots,2^k-1\}^d$,
for which $\Omega_k(p,s)$ is non-empty. We observe, that it is impossible to find an interval $I_{\ell}\subset[0,1]$
with $|I_{\ell}|>\frac{s_{\ell}}{2^k}$ and $\inf I_{\ell}\ge p_{\ell}-\frac{1}{2^k}$ if $p_{\ell}-\frac{1}{2^k}+\frac{s_{\ell}}{2^k}\ge 1$.
Therefore, $\Omega_k(p,s)$ is empty if $p_{\ell}2^k\ge 2^k+1-s_{\ell}$ for some $\ell=1,\dots,d.$

Furthermore, this implies that if $s$ is such that $\Omega_k(p,s)\neq\varnothing$ for some $p\in M_k^d$, then there are exactly
$\prod_{\ell=1}^d(2^k-s_\ell)$ choices for $p$ with
$\Omega_k(p,s)\neq\varnothing$, i.e., $\#\{p:\Omega_k(p,s)\neq\varnothing\}=\prod_{\ell=1}^d(2^k-s_\ell)$.
Denoting $m_1(s):=\#\{\ell\colon s_\ell<2^k-1\}$, we see that
$$
\#\{p:\Omega_k(p,s)\neq\varnothing\}\le 2^{k m_1(s)}.
$$
Now note that, 
for $B\in\Omega_k(p,s)$, we have
\[
2^{-k} \,<\, |B| \,\le\, \prod_{\ell=1}^{d} \Bigl(\frac{s_\ell+1}{2^k}\Bigr)
\,\le\, \Bigl(1-\frac{1}{2^k}\Bigr)^{m_1(s)}.
\]
Therefore, $\Omega_k(p,s)\neq\varnothing$ implies that 
\begin{equation}\label{eq:m1}
m_1(s) \,<\, A_k := \ln(2)\, k\, 2^k. 
\end{equation}
Indeed, we have
$$
2^{-k}<\Bigl(1-\frac{1}{2^k}\Bigr)^{m_1(s)}\quad\text{and}\qquad 
2^{-k}>\Bigl(1-\frac{1}{2^k}\Bigr)^{\ln(2)k2^k}.
$$
The latter of these two formulas follows by a monotone convergence of $(1-1/2^k)^{2^k}$ up to $e^{-1}$.

The number of $s\in\{0,\dots,2^k-1\}^{d}$ with $m_1(s) < A_k$ 
is bounded by 
\[
\binom{d}{A_k}\, 2^{k A_k} \;<\; \left(\frac{4d}{k}\right)^{A_k}, 
\]
where we use $\binom{d}{A_k}\le(e d/A_k)^{A_k}$ and $e/\ln(2)<4$.
We obtain
\[\begin{split}
\#\Bigl\{(p,s)\colon \Omega_k(p,s)\neq\varnothing\Bigr\} 
\;&<\; \left(\frac{4d}{k}\right)^{A_k} 2^{k A_k} 
\;\le\; \exp\Bigl(k\, 2^{k}\Bigl(k + \log_2(4d/k)\Bigr)\Bigr) \\
\;&\le\; \exp\Bigl(k\, 2^{k} \log_2\bigl(2^{k+1}d\bigr)\Bigr)
\end{split}\]
and 
\[\begin{split}
\PP\bigl(\exists B\in\Omega_k\colon X\cap B=\varnothing\bigr)
\;&<\; \exp\Bigl(k\, 2^{k} \log_2\bigl(2^{k+1}d\bigr) -n\,2^{-k-4}\Bigr),
\end{split}\]
which is smaller than one if 
\[
n \;\ge\; 2^4\, k\, 2^{2k} \log_2\bigl(2^{k+1}d\bigr).
\]
This ensures the existence of a set $X$ with $n$ points,
for which $X\cap B\neq\varnothing$ for all cubes $B$ with $|B|>2^{-k}$. 
Therefore,
\[
N(2^{-k},d) \;\le\; 2^4\, k\, 2^{2k} \log_2\bigl(2^{k+1}d\bigr).
\]
Finally, from $2^{-k}\le\varepsilon<2^{-k+1}$, $2^{k-1}<\varepsilon^{-1}\le 2^k$ 
and $k\ge \log_2(1/\varepsilon)>k-1$, we get that


\[\begin{split}
N(\eps,d) \,&\le\, 2^6\,\frac{\bigl(1+\log_2(\varepsilon^{-1})\bigr)
	\,\log_2\bigl(4 d \eps^{-1}\bigr)}{\varepsilon^{2}} \\
\,&\le\, 2^7\,\log_2(d)\,\frac{\bigl(1+\log_2(\varepsilon^{-1})\bigr)^2}{\varepsilon^{2}}.
\end{split}\]
\end{proof}
\medskip

For the proof of Theorem~\ref{thm:main} it remains to prove Lemma~\ref{lem:prob}. 
But before that, we state an alternative bound.

\bigskip

\begin{rem}\label{rem:daniel}
If we repeat the above computations with the bound $m_1(s)<A_k$ replaced by 
$m_1(s)\le d$, then we see that also 
\[\begin{split}
\PP\bigl(\exists B\in\Omega_k\colon X\cap B=\varnothing\bigr)
\;&<\; 2^{2kd}\,\exp\Bigl(-n\,2^{-k-4}\Bigr).
\end{split}\]
As above, this shows that
\[
N(\eps,d) \,\le\, 2^6\,d\,\frac{1+\log_2(\varepsilon^{-1})}{\varepsilon}
\]
for $\eps\in(0,1/2)$. This is, up to constants, the result that was 
proved recently by Rudolf~\cite{Ru17}.
Note that Rudolf's bound is better than ours from Theorem~\ref{thm:main} 
if $A_k>d$, i.e., if 
$\eps<C/(d \ln(d))$ for some $C>0$.
\end{rem}

\bigskip

\begin{proof}[Proof of Lemma~\ref{lem:prob}]
Let $x$ be uniformly distributed in $M_k^d$ and let
$B=I_1\times\dots\times I_d\in\Omega_k$. 
Then we know that $B\in\Omega_k(p,s)$ for some $p\in M_k^d$ and
$s\in\{1,\dots,2^k-1\}^d$. By the definition of $\Omega_k(p,s)$, see \eqref{eq:defOmegak}, we know that
$\inf I_{\ell}<p_{\ell}$ for all $\ell=1,\dots,d.$ On the other hand,
\[
\sup I_{\ell}=\inf I_{\ell}+|I_{\ell}|> p_{\ell}-\frac{1}{2^k}+\frac{s_{\ell}}{2^k}.
\]
We conclude, that $I_{\ell}\cap M_k\supset\Big\{p_{\ell},\dots,p_{\ell}+\frac{s_{\ell}-1}{2^k}\Bigr\}$ for every $\ell=1,\dots,d.$
Therefore, if we set $B(p,s):=\prod_{\ell=1}^d[p_\ell,p_\ell+\frac{s_\ell-1}{2^k}]$,
we obtain $B\cap M_k^d\supset B(p,s)\cap M_k^d$. Let us also observe, that the probability that a randomly selected element of $M_k$
falls into $\Big\{p_{\ell},\dots,p_{\ell}+\frac{s_{\ell}-1}{2^k}\Bigr\}$ is equal to $\frac{s_\ell}{2^k-1}$.
Hence, 
\[
\PP(x\in B) \,\ge\, \PP(x\in B(p,s)) \,=\, \prod_{\ell=1}^{d}\Bigl(\frac{s_\ell}{2^k-1}\Bigr)
\,=\, \prod_{\ell\in D_s}\Bigl(\frac{s_\ell}{2^k-1}\Bigr)
\]
with $D_s:=\bigl\{\ell\in\{1,\dots,d\}\colon s_\ell<2^k-1\bigr\}$. 
Next, we use the inequality
\begin{equation}\label{eq:ineq'}
\Bigl(\frac{j}{2^k-1}\Bigr)
\,\ge\, \Bigl(1-\frac{1}{k2^k}\Bigr)\, \Bigl(\frac{j+1}{2^k}\Bigr)^{\frac{k}{k-1}}
\quad\text{ for all }\quad j=1,\dots,2^k-2,
\end{equation}
which will be proved later.
We obtain
\[\begin{split}
\PP(x\in B)\;&\ge\; \Bigl(1-\frac{1}{k2^k}\Bigr)^{|D_s|}\, \prod_{\ell\in D_s} \Bigl(\frac{s_\ell+1}{2^k}\Bigr)^{\frac{k}{k-1}}\\
&\ge\; \Bigl(1-\frac{1}{k2^k}\Bigr)^{|D_s|}\; |B|^{\frac{k}{k-1}}.
\end{split}\]
Since $|D_s|=m_1(s)<\ln(2)k2^k$ if $B\in\Omega_k(p,s)\neq\varnothing$, 
see~\eqref{eq:m1}, we obtain
\begin{align*}
\PP(x\in B) \;&\ge\; \Bigl(1-\frac{1}{k2^k}\Bigr)^{\ln(2)k2^k} 2^{-\frac{k^2}{k-1}} 
\;\ge\; \Bigl(1-\frac{1}{8}\Bigr)^{8\ln(2)} 2^{-\frac{k^2}{k-1}}\\
\;&>\; \frac{1}{4}\cdot 2^{-\frac{k^2}{k-1}} \;\ge\; \frac{1}{16}\cdot 2^{-k},
\end{align*}
where we have again used the monotonicity of the sequence 
$(1-1/k)^k$ and that $\frac{k^2}{k-1}\le k+2$ for $k\ge 2$.


To show \eqref{eq:ineq'}, we prove that
\begin{equation}\label{eq:ineq''}
\min_{j=1,2,\dots,2^k-2}\frac{j}{(j+1)^\frac{k}{k-1}}\ge (2^k-1)2^{-\frac{k^2}{k-1}}\Bigl(1-\frac{1}{k2^k}\Bigr).
\end{equation}
As the function $t\mapsto \frac{t}{(1+t)^{\frac{k}{k-1}}}$ has only one local extremum on $(0,\infty)$ and this extremum is a local maximum, it is enough to
consider $j\in\{1,2^k-2\}$ and to ensure that
$$
\min\Bigl(\frac{1}{2^{\frac{k}{k-1}}},\frac{2^k-2}{(2^k-1)^{\frac{k}{k-1}}}\Bigr)\ge (2^k-1)2^{-{\frac{k^2}{k-1}}}\Bigl(1-\frac{1}{k2^k}\Bigr).
$$
This splits naturally into two inequalities. The first one (for $j=1)$ follows from
$$
\frac{1}{2^{\frac{k}{k-1}}}= 2^k\cdot 2^{-{\frac{k^2}{k-1}}}.
$$
The second one (for $j=2^k-2$) is equivalent to
$$
\frac{2^k-2}{2^k-1}\ge \Bigl(\frac{2^k-1}{2^k}\Bigr)^{\frac{k}{k-1}}\Bigl(1-\frac{1}{k2^k}\Bigr),
$$
which, by monotonicity, will be established if we prove it with the exponent $\frac{k}{k-1}$ replaced by 1, i.e.,
$$
\frac{2^k-2}{2^k-1}\ge \Bigl(\frac{2^k-1}{2^k}\Bigr)\Bigl(1-\frac{1}{k2^k}\Bigr).
$$
By simple algebraic manipulations, this is equivalent to
$$
2^{2k}-2\cdot 2^{k}\ge (2^k-1)^2\Bigl(1-\frac{1}{k2^k}\Bigr)=2^{2k}-2\cdot 2^k+1-\frac{2^{2k}-2\cdot 2^k+1}{k2^k}
$$
and
$$
2^{2k}-2\cdot 2^k+1\ge k\cdot 2^k,
$$
which holds for $k\ge 2$. 
This finishes the proof of $\PP(x\in B)>2^{-k-4}$ for all $B\in\Omega_k$.

For the second statement of the lemma note that
$B\cap M_k^d\supset B(p,s)\cap M_k^d$ for all $B\in\Omega_k(p,s)$.
This shows
\[
\PP\bigl(\forall B\in\Omega_k(p,s)\colon x\in B \bigr)
\;\ge\; \PP\bigl(x\in B(p,s)\bigr)
\;>\; 2^{-k-4}
\]
and therefore
\[
\PP\bigl(\exists B\in\Omega_k(p,s)\colon x\notin B \bigr)
\;<\; 1-2^{-k-4} \,\le\, \exp\bigl(-2^{-k-4} \bigr).
\]
\end{proof}

\medskip

\begin{rem}
We stress that our proof is very much inspired by the proof of 
Sosnovec~\cite{Sos17}. 
The idea of counting, for each box $B$, the maximal number of coordinates 
with length at least $1-1/2^k$, i.e., $A_k$, is from there.
Our new approach is to consider also the smaller boxes more carefully. 
In particular, note that Lemma~\ref{lem:prob} shows that random points 
from $M_k^d$ ``behave'' like uniformly distributed points from $[0,1]^d$ 
as long as we consider only boxes with volume larger $1/2^k$.
\end{rem}

\begin{rem}
It seems that our technique does not lead to any improvement on the 
upper bounds of Rudolf~\cite{Ru17} in the periodic setting. 
In any case, our present proof would not work. 
For this, note that the proof of Lemma~\ref{lem:prob} requires, 
in particular, that every box $B$ with $|B|>1-1/2^k$ will be 
reached by a random $x\in M_k^d$ with probability one. 
This is clearly not the case if $B$ is allowed to be periodic.
\end{rem}

\goodbreak

\end{document}